\documentclass[12pt]{amsart}

\usepackage{amsmath,amssymb,amsthm}
\usepackage[dvipdfm]{graphicx}
\usepackage{enumerate}
\usepackage[dvips]{color}
\usepackage{version}

\newtheorem{Thm}{Theorem}[section]

\newtheorem{Prop}[Thm]{Proposition}

\newtheorem{``Conj"}[Thm]{``Conjecture"}

\newtheorem{Ques}{Question}

\theoremstyle{remark}
\newtheorem{Rem}[Thm]{Remark}

\theoremstyle{definition}
\newtheorem{Def}[Thm]{Definition}

\newtheorem*{ack}{Acknowledgements}

\begin{document}

\title[K-moduli of Fano varieties]
{Compact moduli spaces of K\"ahler-Einstein Fano varieties}

\author{Yuji Odaka}
\date{Communicated by T.~Mochizuki. Received December 22, 2014. 
February 22, 2015. March 19, 2015. }
\address{Department of Mathematics, Faculty of Science, 
Kyoto University, Kyoto 606-8502, Japan}
\email{yodaka@math.kyoto-u.ac.jp}

\maketitle

\begin{abstract}
We construct geometrically compactified moduli spaces of 
K\"ahler-Einstein Fano manifolds. 
\end{abstract}


\section{Introduction}

In this paper, we construct 
compactified moduli algebraic spaces of Fano manifolds which have 
K\"ahler-Einstein metrics or equivalently 
(thanks to \cite{CDS}, \cite{Tia2}, combined with \cite{Ber},
\cite{Mab1},\cite{Mab2}) 
are K-polystable, following the (precise) conjecture in \cite{OSS} 
formulated with C.~Spotti and S.~Sun. 
The K-stability was originally introduced by G.~Tian \cite{Tia} 
and formulated in a purely algebraic way by S.~Donaldson \cite{Don0}. 
Brief explanations of the definition and the statement of the recent equivalence theorem 
with K\"ahler-Einstein metrics existence 
are given at the beginning of section \ref{precise.Kmod} and the subsection \ref{Kst.def}. 
Roughly speaking, our main result of this paper is: 

\begin{Thm}[Algebro-geometric statement, over $\mathbb{C}$]\label{Main.rough}
For any positive integer $n$, 
there is a ``canonical'' algebraic compactification $\bar{M}$ of 
the moduli space $M$ of K-polystable smooth 
Fano manifolds of dimension $n$, whose boundary paramterises K-polystable 
(kawamata-log-terminal $\mathbb{Q}$-Gorenstein smoothable) 
$\mathbb{Q}$-Fano varieties of the same dimension. 
\end{Thm}

More precisely speaking, the compactification $\bar{M}$ 
is an algebraic space in the sense of Artin \cite{Art} in the above result. 
For most precise meanings, 
see (section \ref{precise.Kmod} and) Theorem \ref{Main.thm}. 
We further expect the compactification to be a
\textit{projective scheme}, following the idea of Fujiki-Schumacher \cite{FS}. 
See the precise expectation in \cite[subsections 3.4, 6.2]{OSS} or our section $2$ (which 
follows \cite{OSS}). 

The corresponding complex differential geometric (roughest) restatement of 
Theorem \ref{Main.rough} is the following. 

\begin{Thm}[Differential geometric re-statement]\label{Main.rough.DG}
The Gromov-Hausdorff compactification of the moduli space of 
K\"ahler-Einstein (smooth) Fano manifolds 
has a structure of compact Hausdorff Moishezon analytic space. 
\end{Thm}

This compactification extends that of the explicit 
$2$-dimensional case study in \cite{OSS}, which was previously and pioneeringly 
proved in the case of complete intersection of $2$ quadric $3$-folds (i.e. degree $4$ del Pezzo 
surfaces) in the old work of Mabuchi-Mukai \cite{MM} much before the 
introduction of K-stability. 

This contrasts to the ``canonically polarised''case (i.e. of ample canonical class) the idea which (for dimensions higher than $1$) goes back to Shepherd-Barron 
\cite{SB}. 
This case was systematically studied by Koll\'ar-Shepherd-Barron \cite{KSB} for surfaces, 
extended by Alexeev to higher dimension \cite{Ale}, 
and now being accomplished with technical details (a book by Professor 
Koll\'ar \cite{Kol} with all the details are being expected to appear). 
Coining the main contributors to the construction in their honors, 
that theory is often attached with 
\textit{Koll\'ar-Sherpherd-Barron-Alexeev}, or with abbreviation ``KSBA''. 

The novel difference is that in our case 
all the varieties parametrised are 
\textit{normal} (even kawamata-log-terminal), hence irreducible, 
while KSBA degenerations are 
usually non-normal as even the simplest case - stable curve \cite{DM} - 
can have up to $3g-3$ components. 

However, in the meantime those two moduli compactifications can be seen 
in a unified point of view, i.e. as 
examples of moduli of K-(semi)stable varieties since the 
semi-log-canonical varieties of ample canonical class is also K-stable by 
\cite{Od1} (``K-moduli'' cf. e.g., \cite[section 5]{Od0}, \cite[Chapter 1]{Spo}). 
Inspired by the breakthroughs \cite{DS} and \cite{Spo}, 
in \cite[Conjecture 6.2]{OSS}, 
the precise formulation of the K-moduli conjecture for Fano varieties case is worked out and 
we will quickly review a part of this in the next section. 

A key technical result 
may be interesting of its own. That is, 
we will establish the following deformation 
picture. The (easier) half of the following statement is proved in 
\cite{OSS} and 
the rest is essentially depending on \cite{LWX}, \cite{SSY} 
which in turn 
use the idea of \emph{Donaldson's continuity method} \cite{CDS}, \cite{Tia2}. 
Our statement is as follows, but we again leave the 
detailed statement to Theorem \ref{local.GIT}. 

\begin{Thm}\label{local.GIT.rough}
If a K\"ahler-Einstein $\mathbb{Q}$-Fano variety $X$ is 
$\mathbb{Q}$-Gorenstein smoothable, 
then in a local $\mathbb{Q}$-Gorenstein (Kuranishi) deformation space of $X$ 
which we denote by ${\it Def}(X)$, the 
existence of K\"ahler-Einstein metric on the corresponding $\mathbb{Q}$-Fano variety is equivalent 
to the GIT polystability of the ${\it Aut}(X)$-action on ${\it Def}(X)$. 
\end{Thm}

As we mentioned, we have already proved in \cite[Lemma 3.6]{OSS} 
that the classical GIT polystability 
of points corresponds to K\"ahler-Einstein $\mathbb{Q}$-Fano varieties, 
which is the easier half of the above theorem \ref{local.GIT.rough}. 
This extends the picture of \cite{Tia}, \cite{Don1} for the commonly studied 
``Mukai-Umemura $3$-fold" case, and
the general result by Szekelyhidi \cite{Sze} which depends 
on the infinite dimensional implicit function theorem. Our proof 
essentially depends on the recent development for one-parameter deformations cases 
in \cite{LWX} and \cite{SSY}. 
We expect that the $\mathbb{Q}$-Gorenstein smoothability condition 
is unnecessary but do not know how to prove in 
that generality, by current technologies. 
It is also related to the list of questions for 
furture in the final section. 

Actually many of the main technical ingredients of the proof are mostly already in 
previous papers 
in this several years 
i.e. \cite{DS}, \cite{Spo}, \cite{Od2}, 
\cite{OSS} and recent \cite{SSY}, \cite{LWX} 
and this paper would not claim elaboration of the essential ideas from before. 

\begin{ack}
This paper originally 
grew out from much more personal and incomplete notes sent to and shared with 
Cristiano Spotti, Song Sun, Chengjian Yao from October $2014$ that is 
three months after when the results of \cite{SSY} 
were informed to the author. 
It was in July of $2014$ during the visit of S.Sun to Kyoto and Tokyo, 
and also there were several seminar 
talks made by them some months before the appearance of \cite{SSY}. 
The author is grateful to all of their neat clarification 
about their results as well as their helpful comments on the draft, 
and would like to say that they also 
made essential contributions 
partially through \cite{SSY} (and \cite{OSS}, \cite{DS}). 
We also thank Jarod Alper for his kind communications 
about sub-section \ref{slice.sec}. 

When the author started to expect ``K-moduli'' \cite[section 5]{Od0}, 
he struggled but could never imagine how to prove even in Fano case and 
the partial proof obtained here just makes clear that he is watching the beauty 
``\textit{on the shoulder of (modern) Giants}'', 
especially for the case of this paper as we do not bring any essentially new idea 
but simply combining the circle of ideas and some standard arguments. 
I would like to take this 
opportunity to thank all the professors, colleagues and friends for the tutorials. 

While finishing the first manuscript of this paper, the author learnt the possibility of 
partial overlap with the revision (to their $2$nd version) 
of \cite{LWX} and our paper. We would like to clarify 
that we worked out independently and both results 
appear on the same day on the internet. 
Finally the author acknowledges the partial support by JSPS Kakenhi 30700356 (Wakate (B)). 
\end{ack}


\section{Precise formulation of K-moduli}\label{precise.Kmod}

In this section, we put the precise formulation of the K-moduli. 
Before that, 
let us recall that the K-stability of a $\mathbb{Q}$-Fano variety $X$ is, roughly speaking, 
defined as positivity of all the Donaldson-Futaki invariants (a variant of the GIT weight) 
associated to every 
one parameter isotrivial degenerations of $X$. We put more precision later in the 
subsection \ref{Kst.def}. The recent development shows the following. 

\begin{Thm}[\cite{CDS}, \cite{Tia2} for smooth $X$, \cite{SSY} for singular $X$]

For any $\mathbb{Q}$-Gorenstein smoothable klt $\mathbb{Q}$-Fano variety $X$, 
the existence of K\"ahler-Einstein metric is equivalent to the K-polystability of $X$. 
\end{Thm}

\noindent
For the definition of K\"ahler-Einstein metrics on singular klt (kawamata-log-terminal) 
$\mathbb{Q}$-Fano varieties, we refer to \cite{Ber} or \cite{SSY} for instance. 

Now we explain our 
precise statement of the K-moduli existence, partially recalling \cite{OSS}. The precision on local 
deformation picture will be put only at the final section (Theorem \ref{local.GIT}). 

For partial self-containedness and the convenience for the readers, 
we recall the notion of 
\textit{KE moduli stack}, introduced for algebraically oriented people. 
We also note that in \cite{OSS}, the notion of 
\textit{KE analytic moduli spaces} 
(for analytic oriented people) was introduced as well. 
For the general theory of algebraic stack, 
we would like to refer to textbooks such as \cite{LM}. 

For those who are not familiar with stacky language, we note that 
algebraic stack (appearing here) is more or less an algebraic scheme 
(such as Hilbert scheme, Chow variety) attached with ``glueing data'' 
which identifies points on the scheme which ``parametrises the same objects''. 
Artin stack is the most general category of algebraic stack, allowing 
``non-discrete automorphism groups'' of the parametrised objects, 
while Deligne-Mumford stack is, roughly speaking, 
for those objects with only discrete automorphism groups. 
The point of introduction of stacky language here is, more or less, 
to make the statement of most precise form with the information on flat families 
of Fano varieties (in concern with K\"ahler-Einstein metrics). 

\begin{Def}[{\cite[Definition 3.13]{OSS}}]\label{KE.stack}
A moduli algebraic (Artin) stack $\bar{\mathcal{M}}$ of $\mathbb{Q}$-Gorenstein 
family of 
$\mathbb{Q}$-Fano varieties is called a \textit{KE moduli stack} if 
\begin{enumerate}
\item there is a categorical moduli algebraic space $\bar{M}$ 
\item it has an \'etale covering $\{ [U_i/G_i] _{i}\}$ of $\bar{\mathcal{M}}$ 
where $U_i$ is affine algebraic scheme 
and $G_{i}$ is some reductive algebraic group, on which there is 
some $G_i$-equivariant 
$\mathbb{Q}$-Gorenstein flat family of $\mathbb{Q}$-Fano varieties. 
\item Closed $G_{i}$-orbits in $U_i$ 
parametrize $\mathbb{Q}$-Gorenstein smoothable K\"ahler-Einstein $
\mathbb{Q}$-Fano varieties via the families of (ii), and via the 
canonical map $\varphi_{i}\colon U_{i}\to \bar{M}$, each such orbit 
maps to a closed point of $\bar{M}$ and every closed point of $\bar{M}$ 
can be obtained in this manner for some $i$. 
\end{enumerate}
We call the coarse algebraic space $\bar{M}$ of $(i)$ a \emph{KE moduli space}. 
If it is an algebraic variety, we also call it \emph{KE moduli variety}.

\end{Def}

Recall that $\bar{M}$ is the \textit{coarse moduli algebraic space} of 
the Artin stack $\bar{\mathcal{M}}$ means that 
there is a morphism $\bar{\mathcal{M}}\rightarrow \bar{M}$ and 
it is universal among the morphisms from $\bar{\mathcal{M}}$ to algebraic spaces. 
In our case, thanks to the condition (ii) and (iii) $\bar{M}$ is also 
set-theoritically ``nice'' i.e. bijectively corrsponds to 
K\"ahler-Einstein $\mathbb{Q}$-Fano varieties. 

For the definition of more differential geometric version ``KE analytic moduli space'', we refer to \cite[Definition 3.14, 3.15]{OSS} since we do not use the notion in this paper and moreover 
it naturally follows from our construction in this paper that $\bar{M}$ satisfies the defining  
conditions of the notion. 

In this paper, we prove the Conjecture 6.2 in \cite{OSS} in $\mathbb{Q}$-Gorenstein smoothable case, 
i.e. those which contain the moduli of smooth Fano manifolds. 

\begin{Thm}[Refined statement of the K-moduli existence]\label{Main.thm}
We fix the dimension of $\mathbb{Q}$-Fano varities in concern, as $n$. 
There is a KE moduli stack $\bar{\mathcal{M}}^{GH}$, in the sense of \cite{OSS}. 
In particular, $\bar{\mathcal{M}}^{GH}$ has a coarse moduli algebraic space 
$\bar{M}$ as a proper separated algebraic space, and $\bar{\mathcal{M}}^{GH}$ is 
good in the sense of Alper \cite{Alp}. 

Then from the Gromov-Hausdorff compactification $M^{GH}
$ (in the sense of \cite{DS}, \cite{OSS}), which is a priori just a compact Hausdorff metric space, there is a homeomorphism $$\Phi\colon \bar{M}^{GH} \rightarrow \bar{M},$$ such that $[X]$ and $\Phi([X])$ parametrize isomorphic $\mathbb{Q}$-Fano varieties 
for any $[X]\in \bar{M}^{GH}$. 
\end{Thm}

\noindent
We remark that the above ``Gromov-Hausdorff" is in the refined sense, that is, 
with care of complex (algebraic) structures as defined and explained in \cite{DS}, \cite{SSY} etc. 


\section{Proof of the main theorems} 

\subsection{Affine \'{e}tale slice in the Hilbert scheme}\label{slice.sec}

We begin the proof of our Main theorem \ref{Main.thm}, 
which will be completed in the end of subsection \ref{local.GIT.sec}. 
In this subsection, we construct an affine slice around $[X]$ inside appropriate Hilbert scheme, 
where $X$ is the $\mathbb{Q}$-Fano variety in concern. In the next subsection, using that slice, 
we formulate and prove the local deformation picture of K\"ahler-Einstein metrics. 

We fix the dimension $n$ of the Fano varieties in concern, and 
consider a finite disjoint union of components of the Hilbert scheme, which we denote by 
${\it Hilb}$, which includes all smooth 
K\"ahler-Einstein Fano manifolds of dimension $n$ and its 
Gromov-Hausdorff limits. Such finite type ${\it Hilb}$ 
exists thanks to the recent breakthrough by 
Donaldson-Sun \cite{DS} and the ``classical'' boundedness result 
by Koll\'ar-Miyaoka-Mori \cite{KMM}. 
In \cite{DS}, it is even proved that 
we can assume that they are all $m$-pluri-anticanonically embedded inside 
$\mathbb{P}^{N}$ with some uniform exponent $m$ and $N=h^{0}(-mK_{X})-1$. 
We work in this setting so that our construction a priori depends on $m$ but we do not 
expect so (see the remark \ref{last.comments} which we put in our revision). 

We set 
${\it Hilb}^{KE}$ as those which parameterize all $m$-pluri-anticanonically 
embedded K\"ahler-Einstein $\mathbb{Q}$-Fano varieities. 
Obviously ${\it Hilb}^{KE}$ is an ${\it SL}(N+1)$-
invariant (equivalently, ${\it PGL}(N+1)$-invariant) subset of ${\it Hilb}$ but note that it does not have a scheme structure in general. 
In fact, as we will show in the next subsection \ref{local.GIT.sec} 
without using the results in this subsection, 
${\it Hilb}^{KE}$ is a \textit{constructible} subset in ${\it Hilb}$. 
So from now on, we replace ${\it Hilb}$ by the Zariski closure 
of ${\it Hilb}^{KE}$ so that we can assume that ${\it Hilb}^{KE}$ 
is dense inside ${\it Hilb}$. From now on, we work inside this 
replaced ${\it Hilb}$. 

Take any point $[X]\in {\it Hilb}^{KE}$. From \cite[III Theorem 4]{CDS}, 
an extension of the Matsushima's theorem \cite{Mat}, 
we know that the automorphism group ${\it Aut}(X)$ 
is a reductive algebraic group. 
Note that ${\it Aut}(X)$ is the isotropy (stabiliser) subgroup 
of the natural ${\it PGL}$-action on ${\it Hilb}$. Thus the 
isotropy subgroup of ${\it SL}$-action on ${\it Hilb}$, which we denote 
as ${\tilde {\it Aut}}(X)$, 
is a central extension of ${\it Aut}(X)$ by $\mu_{N+1}$, the finite 
group of $(N+1)$-th roots of unity isomorphic to $\mathbb{Z}/(N+1)\mathbb{Z}$ which 
acts trivially on ${\it Hilb}$. 
The reason why we think also ${\it SL}$-action not only ${\it PGL}$-action is 
sometimes it is needed to 
make the action available at the level of vector space $H^{0}(X,-mK_{X})$ 
i.e. cone over the projective 
space $\mathbb{P}^{N}$. 

Also let us recall that 
${\it Hilb}\subset \mathbb{P}_{*}(V)$ with some ${\it SL}$-representation $V$ from 
the construction of the Hilbert scheme by Grothendieck. 
(Here $\mathbb{P}_{*}$ denotes covariant projectivisation unlike Grothendieck's 
notation. ) Noting that $[X]$ corresponds to a ${\tilde{\it Aut}}(X)$-invariant 
one dimensional vector space $\mathbb{C}v\subset V$, 
we can decompose ${\tilde{\it Aut}}(X)$'s linear 
representation as $V=\mathbb{C}v\oplus V'$ where 
$V'$ is also ${\tilde {\it Aut}}(X)$-invariant. 
We owe Jarod Alper for the clarification about this and be grateful to him. 
This is possible since 
we know ${\tilde{\it Aut}}(X)$ is 
reductive. Then we can take an ${\it Aut}(X)$-invariant open subset $U_{[X]}$ 
as ${\it Hilb}\setminus \mathbb{P}_{*}(V')$. This is also affine since 
$\mathbb{P}_{*}(V')$ is an ample divisor of the original projective space 
$\mathbb{P}_{*}(V)$. 

Note that this open neighborhood $U_{[X]}$ of $[X]$ 
is only ${\it Aut}(X)$-invariant (or equivalently $\tilde{\it Aut}(X)$-invariant),  
but \textit{not} necessarily ${\it SL}$-invariant. 
In the meantime, this affine-ness of $U_{[X]}$  
enables us to apply the following techniques of taking 
\'etale slice mainly due to \cite{Luna} (a.k.a., Luna's ``\textit{\'etale slice theorem}'' cf. \cite[5.3]
{Dre}). 
We include the short outline of the proof for the readers' convenience, partially 
because we slightly extend original theorem of \cite{Luna}, 
but basically the argument below is from the nice 
exposition of \cite{Dre} on the Luna's theory \cite{Luna}. 
First we can easily construct a closed immersion of $U_{[X]}$ into an 
${\it Aut}(X)$-acted smooth affine space 
$\tilde{U}_{[X]}$ (cf. e.g, \cite[Lemma 5.2]{Dre}) with the same embedded dimension 
of $[X]\in U_{[X]}$. 
Then for the proof of \'etale slice theorem of \cite{Luna} (cf., e.g., \cite[Lemma 
5.1]{Dre}), 
it is proved that there is an ${\it Aut}(X)$-equivariant affine regular map 
$\varphi \colon \tilde{U}_{[X]}\rightarrow (T_{[X]}U_{[X]})$ which is 
\'etale at $[X]$. This is again depending on the reducitivity of ${\it Aut}(X)$. 
We use this equivariant map as follows. 

We decompose the ${\it Aut}(X)$-representation $T_{[X]}U_{[X]}$ as 
$T_{[X]}({\it SL}(N+1)[X]\cap U_{[X]})\oplus N$ with some ${\it Aut}(X)$-invariant 
subvector space $N$. Then we define $V_{[X]}:=(\varphi^{-1}N\cap U_{[X]})\subset 
U_{[X]}$, which is an ${\it Aut}(X)$-invariant 
locally closed affine subset of ${\it Hilb}$ including $[X]$. 
Then this $V_{[X]}\subset U_{[X]}$ is an \'{e}tale slice in the sense of 
\cite{Luna, Dre}, in particular $[V_{[X]}/{\it Aut}(X)]\to [U_{[X]}/{\it PGL}]$ 
is an \'etale morphism (between two quotients stacks). 
We omit more details and the rest of the proof of this known fact since it 
simply follows from the proof of \cite[Theorem 5.3]{Dre} or 
\cite[subsection 2.2]{AK}. 


\subsection{K-stability via CM line bundle}\label{Kst.def}

Before proceeding to next arguments, we briefly recall the fundamental relation of the 
K-stability and \textit{the CM line bundle} (\cite{FS, PT}), 
which we regard as a definition of the K-stability in this paper. 

The CM line bundle, in our setting, is a certain ${\it SL}$-equivariant line bundle $\lambda_{CM}$ on ${\it Hilb}$ (\cite{FS}, \cite{PT}, \cite{FR}). As the actual construction is a little complicated and 
we do not need in this paper, we omit its details and refer to \cite{FR}. 

In our setting, for the given positive integer parameter $m$, 
the $K_{(m)}$-stability of $\mathbb{Q}$-Fano varieties means 
the following (as in \cite{Od2}, just following \cite{Don0}). 

\begin{Def}
As in the previous subsection, 
suppose that a (klt) $\mathbb{Q}$-Fano variety $X$ satisfies that $-mK_{X}$ is a 
very ample line bundle ($m\in \mathbb{Z}_{>0}$). Then the $\mathbb{Q}$-Fano variety 
$X$ (more precisely, $(X,-K_{X})$) is said to be $K_{(m)}$-stable 
if for any nontrivial one parameter subgroup $f\colon \mathbb{C}^{*}\to {\it SL}$, 
minus the weight of $\lambda_{CM}|_{{\it lim}_{t\to 0}(f(t)[X])}$ (called as the Donaldson-Futaki invariant associated to $f$) is positive. The one parameter degeneration of $X$ along 
$\overline{f(\mathbb{C}^{*})\cdot [X]}\subset {\it Hilb}$ is called 
``\textit{test configuraiton}" by \cite{Don0}. )

Similarly, $X$ is said to be $K_{(m)}$-semistable (resp. $K_{(m)}$-polystable) if 
all the Donaldson-Futaki invariants are non-negative (resp. $X$ is semistable and 
the Donaldson-Futaki invariant of $f$ is positive if and only if the orbit closure 
$\overline{f(\mathbb{C}^{*})\cdot [X]}\subset {\it Hilb}$ is contained 
in the {\it SL}-orbit of $X$ (such a 
degeneration is called ``\textit{product} test configuration")). 

$X$ is said to be K-stable (resp. K-semistable, K-polystable) if it is $K_{(m)}$-stable 
 (resp. $K_{(m)}$-semistable, $K_{(m)}$-polystable) 
for all sufficiently divisible positive integer $m$. 
\end{Def}

\subsection{Local GIT polystability}\label{local.GIT.sec}

In this subsection, we apply \cite[Lemma 3.6]{OSS} to the ${\it Aut}(X)$-action on the Affine \'{e}tale slice $V_{[X]}$ and see that 

\begin{quotation}
the points 
corresponding to some 
K\"ahler-Einstein $\mathbb{Q}$-Fano varieties are GIT polystable in $V_{[X]}$ 
with respect to the ${\it Aut}(X)$-action, 
\end{quotation}
and we denote the polystable locus in $V_{[X]}$ as $V_{[X]}^{{\it ps}}$. 
The following theorem shows that the converse to \cite[Lemma 3.6]{OSS} also holds in 
appropriate sense, and later on this will be crucial for us. 

\begin{Thm}[Local deformation picture of KE Fano varieties]\label{local.GIT}
For small enough affine \'etale slice $V_{[X]}$, i.e. after shriking $V_{[X]}$ to 
${\it Aut}(X)$-invariant open affine neighborhood of $[X]$ if necessary, we have 
$$V_{[X]}^{ps}=V_{[X]}\cap {\it Hilb}^{KE}.$$ 

Recall that $V_{[X]}^{ps}$ denotes the GIT poly-stable locus of the affine slice 
$V_{[X]}$ in the Hilbert scheme, with 
respect to the ${\it Aut}(X)$-action. 
\end{Thm}

It roughly says that, \'etale locally, the existence of K\"ahler-Einstein metrics on 
$\mathbb{Q}$-Fano varieties is equivalent to the classical GIT polystability, 
at least in the $\mathbb{Q}$-Gorenstein smoothable case (we expect this is the case 
in non-smoothable case as well). Note that the above statement is about ``local'' 
deformation picture in the sense we need to shrink $V_{[X]}$ in general. Otherwise 
the statement is false and indeed the proof requires that shrinking. 

This refines \cite[section 7]{Tia}, \cite[subsection 5.3]{Don1} which 
treated Mukai-Umemura (Fano) $3$-folds, $\mathbb{Q}$-Fano varieties case of 
\cite{Sze} 
and of course \cite[Lemma 3.6]{OSS}. We expect that this will be a 
fundamental tool in the further 
study of K\"ahler-Einstein metrics on $\mathbb{Q}$-Fano 
varieties in future. 

\begin{proof}[proof of Theorem \ref{local.GIT}]

The one side that
$V_{[X]}\cap {\it Hilb}^{KE}\subset V_{[X]}^{ps}$ 
is exactly (a special case of) \cite[Lemma 3.6]{OSS} and here is the argument for 
the 
other side i.e. $$V_{[X]}^{ps}\subset V_{[X]}\cap {\it Hilb}^{KE}.$$ 
We prove that this holds, once we replace $V_{[X]}$ with small enough affine 
${\it Aut}(X)$-invariant slice if necessary.

Note that the difference set $V_{[X]}^{ps}\setminus (V_{[X]}\cap {\it Hilb}^{KE})$ 
is constructible since both 
$V_{[X]}^{ps}$ and $(V_{[X]}\cap {\it Hilb}^{KE})$ are constructible subsets. 
The constructibility of polystable locus is a 
standard fact in the Geometric Invariant 
Theory. 
We now explain how to show the constructibility of 
$(V_{[X]}\cap {\it Hilb}^{KE})\subset V_{[X]}$. 
Indeed, due to \cite[Theorem 1]{SSY}, we know the equivalence of K-polystability and 
existence of K\"ahler-Einstein metrics for $\mathbb{Q}$-Gorenstein 
smoothable Fano varieties in general. Moreover, combining 
\cite[esp. II Theorem1, III Theorem 2]{CDS}, \cite[4.2.2]{SSY} and 
the arguments of \cite[esp. (2.4-8)]{Od2}, 
we know that it is also equivalent to the quantised ``$K_{(m)}$-polystability'' 
in the above sense of subsection \ref{Kst.def} for sufficiently divisible uniform $m\gg 0$ i.e. 
we can bound the exponent $m$ for testing K-(poly)stability. 
For the readers' convenience, we recall from \cite[esp. (2.4-8)]{Od2} 
that the main point of the 
uniform bound $m$ was the uniform positive lower bounds of (small) 
\textit{angles} of conical K\"ahler-Einstein metrics on all the $\mathbb{Q}$-Fano 
varieties parametrised in ${\it Hilb}$. 

Then the proof of 
the constructibility of the $K_{(m)}$-polystable locus inside ${\it Hilb}$ 
follows from the arguments in \cite[esp. (2.10-12)]{Od2} 
only with the additional but simple concern 
whether the test configurations are of product type or not. 

To prove the theorem, we suppose the contrary and get contradiction. 
So let us suppose that \textit{for any small enough} 
affine ${\it Aut}(X)$-invariant slice 
$V_{[X]}$ of $[X]$, we have $V_{[X]}^{ps}\neq V_{[X]}\cap {\it Hilb}^{KE}$.

Therefore, from our assumption that
$V_{[X]}^{ps}\neq V_{[X]}\cap {\it Hilb}^{KE}$, 
we have an irreducible locally closed subvariety $W$ inside 
the difference subset $V_{[X]}^{ps}\setminus (V_{[X]}\cap {\it Hilb}^{KE})$ 
whose closure meets $[X]$ and we take a sequence $P_{i}$ in $W$ converging to 
$[X]$. Otherwise, we can shrink $V_{[X]}$ to make it satisfies 
$V_{[X]}^{ps}= V_{[X]}\cap {\it Hilb}^{KE}$. 
Now we fixed our slice $V_{[X]}$. 

We take any ${\it SL}$-equivariant compactification of the algebraic group 
${\it SL}$ (such as \cite{DP}, or apply \cite{Sum}) 
and denote it with $\bar{\it SL}$ and 
consider the rational map 
$\varphi\colon\overline{V_{[X]}}\times \bar{{\it SL}}\dashrightarrow 
{\it Hilb}$ induced by the ${\it SL}$-action. Here 
$\overline{V_{[X]}}$ denotes the Zariski closure of $V_{[X]}$ inside 
${\it Hilb}$. 
Then we take a 
${\it SL}$-equivariant resolution of indeterminancy of $\varphi$ 
as 
$$\tilde{\varphi}\colon T\rightarrow {\it Hilb}.$$ 

\noindent
So $T$ is a certain ${\it SL}$-equivariant blow up of 
$\overline{V_{[X]}}\times \bar{{\it SL}}$ along some ideal co-supported 
on $\overline{V_{[X]}}\times (\bar{SL}\setminus {\it SL})$. 
Via the morphism from $T$ to ${\it Hilb}$, 
we can regard $T$ as a parameter space of 
Fano varieties and its degenerations.

Then take sequences $P_{i}\in W\subset \overline{V_{[X]}}\simeq \overline{V_{[X]}}
\times \{e\}\subset T
(i=1,2,\cdots)$ 
which converges to $[X]\in V_{[X]}$ and 
$P_{i,j} \in V_{[X]}\simeq V_{[X]}\times \{e\}\subset T ({i,j=1,2,\cdots})$, parametrising smooth K\"ahler-Einstein Fano manifolds $X_{i,j}$, which converges to 
$P_{i}$ when $j$ goes to infinity. 

Thanks to \cite{DS}, we know that (by taking subsequence) the Gromov-Hausdorff limit of 
$X_{i,j}$ with K\"ahler-Einstein metrics exists as another K\"ahler-Einstein 
$\mathbb{Q}$-Fano variety $Y_{i}$. Furthermore, from their construction as a 
limit inside the Hilbert scheme (cf., \cite[Theorem 1.2]{DS}), we know that 
there is a sequence of elements of ${\it SL}$ which we denote by 
$\phi_{i,j}$ such that ${\it lim}_{j\to \infty}\phi_{i,j}(P_{i,j})$ 
represents a point $Q_{i}$ which parametrises the 
($m$-th pluri-anticanonically embedded) K\"ahler-Einstein Fano variety $Y_{i}$, 
for each fixed $i$. By the standard diagonal argument, it also follows from 
\cite{DS} that ${\it lim}^{GH}_{i\to \infty}Y_{i}$ exists (limit in the (refined) 
Gromov-Hausdorff 
sense as in \cite{DS}) as yet another K\"ahler-Einstein $\mathbb{Q}$-Fano variety 
$Y$ where the corresponding point will be denoted by $Q\in {\it Hilb}$. 
As the blow up morphism $T\rightarrow \overline{V_{[X]}}\times \bar{{\it SL}}$ 
is (topologically) a proper morphism, we can take all these points in $T$. 

Our general idea is to apply (recently obtained) 
separated-ness theorem to the two ``degenerations'' 
of $X_{i,j}$ to $[X]$ and $[Y]=Q\in T$, both of which parametrise 
K\"ahler-Einstein $\mathbb{Q}$-Fano varieties. 
To put precision on the idea, from now on, 
we proceed to some more algebro-geometric arguments. 






Set $T^{o}$ as the (open dense) subset of $T$ which is the 
preimage of ${\it SL}\subset \bar{\it SL}$. We also set 
$\partial T:=T\setminus T^{o}$. 
Consider some general affine curve $C\subset T$ which passes through $Q$ 
and intersects $\partial T \cup (V_{[X]}^{ps}\setminus (V_{[X]}\cap {\it Hilb}^{KE}))$ only at the point $\{Q\}$. 

On the other hand, take the natural retraction 
$r\colon T^{o}\rightarrow \overline{V_{[X]}}$ 
induced by ${\it SL}\rightarrow \{e\}$ where $e\in {\it SL}$ is the unit of 
special linear group ${\it SL}$ and partially complete  
$C'^{0}:=r(C\setminus\{Q\})$ naturally to $C'$ with $i\colon C\simeq C'$. 
Note that from the construction, $r$ also naturally extends to a morphism 
$$
\tilde{r}\colon T\rightarrow {\it Hilb}
$$
from the whole $T$. 
Then from our construction, 
the image $i(Q)$ is nothing but 
the original $[X]\in {\it Hilb}$. We can see it as follows. 
Since $i$ should preserve the image of $\tilde{r}$, 
$\tilde{r}(i(Q))=\tilde{r}(Q)$ and that 
$\tilde{r}(Q)=\tilde{r}({\it lim}_{i\to \infty}(Q_{i}))
=\tilde{r}({\it lim}_{i\to \infty}({\it lim}_{j\to \infty}(P_{i,j}))
={\it lim}_{i\to \infty}(\tilde{r}(P_{i}))=[X]$. The last equality follows from 
our construction of $P_{i}$. 
(Here all the limit symbols are in the usual sense of analytic topology). 

The crucial result we need from now on is the following. 
Although we do not have any contribution on it 
in this paper, we would like to recall the 
result as we need a comment (on how to combine \cite{LWX},\cite{SSY},\cite{CDS}, 
as written below) on the proof to make things rigorous. 
I thank S.Sun for the mathematical clarification of this point.

\begin{Thm}[{\cite[Thm1.1 of v1]{LWX}+\cite[Thm1.1]{SSY},\cite{CDS}}]\label{sep}
Let $\mathcal{X}$ and $\mathcal{Y}$ be two
$\mathbb{Q}$-Gorenstein flat deformations of K\"ahler-Einstein 
$\mathbb{Q}$-Fano varieties 
over a smooth curve $C\ni 0$. Suppose $\mathcal{X}_t\ \cong\
\mathcal{Y}_t
$ for $t\neq 0$ 
and further that these are all smooth (i.e. generically smooth).  If $\mathcal{X}_0$ and $\mathcal{Y}_0$ are both 
K-polystable, then they are isomorphic $\mathbb{Q}$-Fano varieties.
\end{Thm}

This follows from the combination of \cite[v1]{LWX} and \cite[Theorem 1.1]{SSY}. 
Note that for \emph{separateness}, \cite[Corollary 1.2]{SSY} needs to assume that $\mathcal{X}_0$ and $\mathcal{Y}_0$ have discrete automorphism groups, while \cite[Remark 6.11]{LWX} needs to assume that $\mathcal{X}_0$ and $\mathcal{Y}_0$ have reductive automorphism groups.  But from \cite[Theorem 1.1]{SSY} we know both $\mathcal{X}_0$ and $\mathcal{Y}_0$ admit KE metrics, so satisfy the assumption on \emph{reductivity} of \cite[v1]{LWX} by \cite[III, Theorem 4]{CDS}. 
(The author had once attempted to prove this \emph{separateness} with 
Professor Richard Thomas but the arguments had a 
technical gap. ) 

We apply the theorem above to the two families of $\mathbb{Q}$-Fano varieties 
corresponding to $C\subset T$ and $C'\subset T$. Then we can show that 
$Q$ is in the ${\it SL}$-orbit of $[X]\in {\it Hilb}$, hence in $T^{o}$ in 
particular. Recall that $Q$ was defined as the limit of $Q_{i}$. 
Hence for $i\gg 0$, $Q_{i}$ is also in $T^{o}$. 
Then it implies that by \cite[Lemma 3.6]{OSS}, 
$i(Q_{i})\in V_{[X]}$, which is well-defined, is GIT polystable 
with respect to the action of the automorphism group 
${\it Aut}(X)$. 

Then we get a contradiction from the general theory of Geometric Invariant 
Theory \cite{GIT} since $i(Q_{i})$ and $P_{i}$ are both 
GIT polystable, while being the limits of sequences which parametrises the 
same polystable point. This completes the proof. 
\end{proof}

\begin{Prop}\label{homeo}
Let $X$ be an arbitrary 
$\mathbb{Q}$-Gorenstein smoothable K\"ahler-Einstein $\mathbb{Q}$-Fano 
variety and denote the corresponding point in the Hilbert scheme 
as $[X]$ which represents 
$m$-pluri-anticanonically embedding $[X]\in {\it Hilb}^{KE}$. 
Then there is a small enough affine ${\it Aut}(X)$-invariant slice $V_{[X]}$ 
of the natural ${\it PGL}$-action on ${\it Hilb}$ such that 
an open neighborhood (in analytic topology) of $\bar{[X]}$ in 
the GIT (categorical) quotient $V_{[X]}//{\it Aut}(X)$ naturally maps 
homeomorphically to $\bar{M}^{GH}$ 
(which eventually becomes an 
\'etale algebraic morphism with the algebraic structure on the latter). 

Analytically speaking, this is equivalent to saying 
that there is an open subset $W$ of $[X]$ 
in $V_{[X]}$ and an analytically open neighborhood $N$ of $[X]\in \bar{M}^{GH}$ 
such that there is a natural homeomorphism 
$$N\rightarrow (W\cap V_{[X]}^{ps})/{\it Aut}(X),$$
preserving the $\mathbb{Q}$-Fano varieties being parametrised. 
\end{Prop}

\begin{proof}[proof of Proposition \ref{homeo}] 
The continuity from $N$ to $(W\cap V_{[X]}^{ps})/{\it SL}$ follows from 
Donaldson-Sun \cite[(proof of) Theorem 1.2]{DS}. The quotient space 
${\it Hilb}^{KE}/{\it SL}$ satisfies the Hausdorff axiom due to the 
separated-ness theorem \ref{sep} proved by (\cite{LWX}+\cite{SSY}) while 
$\bar M^{GH}$ is compact due to the Gromov compactness theorem. It is a general 
theorem that continuous bijection from a compact topological space 
(now $\bar M^{GH}$) to 
a Hausdorff space (now ${\it Hilb}^{KE}/{\it SL}$) is automatically 
homeomorphism. 
\end{proof}

Summarising the above discussions, we conclude the proof of 
our main theorem \ref{Main.thm}, the moduli construction, as follows. 

\begin{proof}[proof of Theorem \ref{Main.thm}]
For each $[X_{i}]\in {\it Hilb}^{KE}$, i.e. $X_{i}$ is 
one of smooth K\"ahler-Einstein Fano $n$-dimensional manifolds or 
one of their Gromov-Hausdorff limits (hence $\mathbb{Q}$-Fano varieties 
with K\"ahler-Einstein metrics by \cite{DS}), 
let us consider $V_{[X_{i}]}$ constructed in 
the subsection \ref{slice.sec}. We replace $V_{[X_{i}]}$ 
by its open ${\it Aut}(X_{i})$-invariant 
open neighborhood, if necessary, 
to make it satisfy the requirement in Theorem \ref{local.GIT}. 

Note that for each $X_{i}$, ${\it PGL}\cdot V_{[X_{i}]}$ is a Zariski open subset 
in ${\it Hilb}$. It follows from the fact that since we constructed 
$V_{[X_{i}]}\subset U_{[X_{i}]}$ as an \'etale 
slice, ${\it PGL}\times_{{\it Aut}(X_{i})}V_{[X_{i}]}\to {\it Hilb}$ 
is an \'etale morphism so in particular an open morphism. 
Thus by quasi-compactness of 
${\it Hilb}$, we only need finitely many $i$ such sets ${\it PGL}\cdot V_{[X_{i}]}$ to 
cover ${\it Hilb}^{KE}$.

We note that $\varphi_{i}\colon [V_{[X_{i}]}/{\it Aut}(X_{i})]\rightarrow [{\it Hilb}/{\it PGL}]$ 
is an \'etale morphism between two quotient stacks, since again the morphism 
${\it PGL}\times _{{\it Aut}(X_{i})}V_{[X_{i}]}\to U_{[X_{i}]}\subset {\it Hilb}$ 
is strongly \'etale (in the sense of \cite[subsection 1.1]{Dre}). 
Please note that it is a priori \textit{not}  
necessarily open immersion (of algebraic stacks) because the slice 
$V_{[X_{i}]}$ is just an \textit{\'etale} slice. 
Glueing together $[V_{[X_{i}]}/{\it Aut}(X_{i})]$ via $\varphi_{i}$s 
which is by definition possible 
inside $[{\it Hilb}/{\it PGL}]$, 
we obtain $[W/{\it PGL}]$ with $W=\cup_{i} ({\it PGL}\cdot V_{[X_{i}]})\subset 
{\it Hilb}$, 
a moduli Artin stack which we denote as $\bar{\mathcal{M}}$. 
Furthermore, as the property \cite[subsection 1.1 (ii)]{Dre} 
of the \'etale slice $V_{[X_{i}]}$ (cf., also \cite[5.3]{Dre}) shows, 
categorical quotients $V_{[X_{i}]}/{\it Aut}(X_{i})$ glue together to form 
a coarse moduli algebraic space $\bar{M}$ of the Artin stack $\bar{\mathcal{M}}$. 

The fact that it is a KE moduli stack in the sense of Definition \ref{KE.stack}
(\cite{OSS}) now follows from Theorem \ref{local.GIT}. Indeed the condition $(iii)$ 
of Definition \ref{KE.stack} is exactly the statement of Theorem \ref{local.GIT} 
and we have proved the condition $(i)$ of Definition \ref{KE.stack} above. The 
remaining $(ii)$ of (\ref{KE.stack}), which says that the 
flat family on $V_{[X_{i}]}$ is $\mathbb{Q}$-Gorenstein flat family (once we 
shrink $V_{[X_{i}]}$ enough), can be easily checked as follows. 
(Please also see \cite[(2.4)]{OSS} for essentially the same arguments. )
Actually in general if we have a point $[X]$ in ${\it Hilb}$ 
corresponding to some normal variety $X$, its deformation parametrised in a 
neighborhood in ${\it Hilb}$ is automatically $\mathbb{Q}$-Gorenstein deformation. 
We set the locus of ${\it Hilb}$ which parametrises normal varieties as 
${\it Hilb}_{\text{normal}}\subset {\it Hilb}$, that is automatically open subset as
it is well known. We denote its subset which parametrises singular (but normal) 
varieties as ${\it Hilb}_{\text{normal.singular}}$. 
Let us take a log resolution of singularities of the pair 
$({\it Hilb}_{\text{normal}},{\it Hilb}_{\text{normal.singular}})$ after Hironaka, 
as 
$f\colon S\to {\it Hilb}$ so that $f^{-1}({\it Hilb}_{\text{normal.singular}})$ is 
a (simple normal crossing) Cartier divisor $\Sigma$ of $S$. Then 
we have a flat projective family 
$\pi\colon (\mathcal{X},\mathcal{O}_{\mathcal{X}}(1))\rightarrow S$ and 
\begin{equation}\label{lin.eq}
\mathcal{O}_{\mathcal{X}}(1)|_{\mathcal{X}\setminus \pi^{-1}(\Sigma)}
\sim_{(S\setminus \Sigma)} \mathcal{O}_{(\mathcal{X}\setminus \pi^{-1}(\Sigma))}(-mK_{\mathcal{X}\setminus \pi^{-1}(\Sigma)}). 
\end{equation}
The above (\ref{lin.eq}) 
implies that there are Weil divisors $\mathcal{D}, \mathcal{D}'$ of $\mathcal{X}$ with 
$\mathcal{O}_{\mathcal{X}}(\mathcal{D})=\mathcal{O}_{\mathcal{X}}(1), \mathcal{O}_{\mathcal{X}}(\mathcal{D}')=
\mathcal{O}_{\mathcal{X}}(-mK_{\mathcal{X}})$ (the latter is only a reflexive sheaf), 
which satisfies that $\mathcal{D}-\mathcal{D}'$ supports on $\pi^{-1}(\Sigma)$. 
In the meantime, any (a priori Weil-)divisor supported 
on the central fiber is a pull back of (Cartier) divisor of $S$ supported on 
$\Sigma$ since all the fibers of $\pi$ are irreducible now. Hence, we 
get $\mathcal{O}(1)\sim_{C}\mathcal{O}(-mK_{\mathcal{X}})$. 

Furthermore, the subset ${\it Hilb}_{\text{klt}}$ 
of ${\it Hilb}_{\text{normal.singular}}$ which parametrises 
(kawamata-)log-terminal varieties is a Zariski open subset, which follows from the 
arguments of \cite[Appendix A]{AH} 
(even easier than that since we only treat normal varieties). In particular, 
$V_{[X_{i}]}$ only parametrises $\mathbb{Q}$-Fano varieties, since each variety 
parametrised in $V_{[X_{i}]}$ has some isotrivial degeneration to a variety 
parametrised in $V_{[X_{i}]}^{ps}$ which is automatically a $\mathbb{Q}$-Fano 
variety. Summarising up, we proved the 
assertion ($(ii)$ of Definition \ref{KE.stack}). 

The topological space structure part is proved in Proposition \ref{homeo}. 
Indeed, note that Proposition \ref{homeo} shows that the Gromov-Hausdorff compactification 
$\bar{M}^{GH}$ is homeomorphic to the coarse moduli space $\bar{M}$ constructed 
above. In particular it shows $\bar{M}$ satiefies the Hausdorff second axiom 
(essentially follows from \cite{CDS}+\cite{LWX}(v1)+\cite{SSY} cf., Theorem \ref{sep}). 
So we complete the proof of Theorem \ref{Main.thm}. 
\end{proof}

\begin{Rem}\label{last.comments}
\textit{This remark is newly put in our revision which appears in the $20$th of March, 2015. }
Our construction of the moduli stacks $\bar{\mathcal{M}}$ and their 
coarse moduli spaces $\bar{M}$ 
a priori depend on the positive integer parameter $m$ (please recall that we consider the $m$-th 
pluri-anti-canonical polarisation of the $\mathbb{Q}$-Fano varieties). 
However we strongly believe 
that they actually do \textit{not} depend on the sufficiently divisible $m$. 
Indeed, we can prove it under the following two hypotheses. 
To the best of the author's knowledge (as of March, 2015) full proofs 
of the hypotheses below are not available yet, although the revision ($2$nd version) of \cite{LWX} 
have partial affirmative results (cf., their section 7) in this direction. 

\begin{enumerate}
\item The K-semistability is an \textit{open condition} for any $\mathbb{Q}$-Gorenstein flat projective family of $\mathbb{Q}$-Fano ($\mathbb{Q}$-Gorenstein smoothable) varieties. 
\item  For any ($\mathbb{Q}$-Gorenstein smoothable) K-semistable 
$\mathbb{Q}$-Fano variety, say $X$, it has a test configuration whose central fibre is a KE 
$\mathbb{Q}$-Fano variety $Y$ (which is K-polystable by \cite{Ber}). 
\end{enumerate}

Our proof of the desired 
$m$-independence of our moduli $\bar{\mathcal{M}}$ and $\bar{M}$, under the 
hypotheses, is simple as follows. 
The above hypotheses 
imply that $W$ coincides exactly with the (open) locus of 
\textit{$\mathbb{Q}$-Gorenstein smoothable 
K-semistable $\mathbb{Q}$-Fano varieties}, 
which we denote as ${\it Hilb}^{sss}$. We prove it as follows. 
Recall that each $\mathbb{Q}$-Fano variety 
corresponding to a point of $W$, isotrivially degenerates to a KE $\mathbb{Q}$-Fano variety 
parametrises in ${\it Hilb}^{KE}$ by our Theorem \ref{local.GIT} and the standard GIT. 
That fact, combined with the first hypothesis (i) implies $W\subset {\it Hilb}^{sss}$. 
On the other hand, (ii) and \cite{DS} (especially their uniform bound of ``$k$") 
imply ${\it Hilb}^{sss}\subset W$ straightforwardly. Thus 
our KE moduli stack $\bar{\mathcal{M}}$, which is isomorphic to the quotient stack $[W/{\it PGL}]$ whose definition involved $m$, is exactly the moduli Artin stack of $\mathbb{Q}$-Gorenstein flat projective family of 
K-semistable $\mathbb{Q}$-Gorenstein smoothable $\mathbb{Q}$-Fano varieties of dimension $n
$. It is this universality which automatically implies that the moduli stacks 
$\bar{\mathcal{M}}$ do not depend 
on the integer $m$. In particular, their 
coarse moduli spaces $\bar{M}$ also do not depend on $m$. 
 
We also make a brief mathematical remark in this revision (March, 2015) for the readers'  
convenience, about the mathematical relation with the $2$nd version of \cite{LWX}. 
It is that the moduli space constructed in the 
$2$nd version of \cite{LWX} is the semi-normalisation of reduced subscheme of our moduli. 
\end{Rem}


\section{For future}

It may be needless to mention but 
the author would like to note that there are quite a lot of interesting 
problems to do 
from now on the K-moduli of Fano varieties, and we list some main of them possibly 
with my personal biase. 
Most of them (perhaps other than Question \ref{Q2}) are natural and 
being shared among the community of this subject and we just write down for the 
record. 

\begin{Ques}\label{Q1}
How about \textit{concrete} examples of $\mathbb{Q}$-Fano varieties? 
\end{Ques}

As far as the author knows, the only fully settled case is \cite{MM},\cite{OSS} 
which are for ($\mathbb{Q}$-Gorenstein smoothable) Del Pezzo surfaces. 
The author would guess \cite[Lemma 3.6]{OSS} and our Theorem \ref{local.GIT} 
will be one of the key tools for this 
direction. For example, the author is tempted to expect 
that many of the standard GIT 
moduli spaces of hypersurfaces, such as cubic $3$-folds and 
$4$-folds case (\cite{All}, \cite{Laza}, \cite{Yok1}, \cite{Yok2}), 
are examples of our K-moduli spaces 
(cf., \cite[Theorem 3.4 and subsection 4.2]{OSS}). The last prediction is 
partially inspired by discussions with Julius Ross. 

\begin{Ques}\label{Q2}
How to construct Gromov-Hausdorff limit of K\"ahler-Einstein Fano manifolds 
(and the K-moduli construction) in purely \textit{algebraic} way? 
\end{Ques}

It is natural to expect that the (refined) GH limit, in the sense of \cite{DS},\cite{OSS} etc, 
is simply equivalent to K-polystable 
limit and then, partially inspired by \cite{LX}, \cite[last section]{Od2} (etc), 
characterised by the minimality of 
the degree of (family version of) Donaldson-Futaki invariant. 
And we further expect that the construction 
will essentially need the idea and theory of the Minimal Model Program. 

\begin{Ques}\label{Q3}
How about \textit{non-smoothable} $\mathbb{Q}$-Fano varieties? 
\end{Ques}

This is the much more general case, 
morally about the moduli space all of whose members parametrise 
\textit{singular} (log-terminal) $\mathbb{Q}$-Fano varieties. 
At this moment, we (and \cite{SSY},\cite{LWX} etc.) all heavily depend on 
the ($\mathbb{Q}$-Gorenstein) smoothability of Fano varieties in concern, 
in order to apply \cite{CDS}, \cite{Tia2} which are for smooth Fano 
\textit{manifolds}. 
But as many algebraically oriented people agree as 
they told, it is natural to expect the completely same picture for 
\textit{general $\mathbb{Q}$-Fano varieties}. 

\begin{Ques}\label{Q4}
What about the \textit{projectivity} of our moduli space? 
\end{Ques}

The expectation is that ``descended'' $\mathbb{Q}$-line bundle from the CM line 
bundle \cite{FS}, \cite{PT} 
explained (with the proof of descending phenomenon) 
at the end of \cite{OSS}, will be \textit{ample} on the coarse compact 
moduli space $\bar{M}$, ensuring the projectivity. 
The expectation is based on the general 
\textit{Weil-Petersson} metrics as in \cite{FS}. Indeed by \cite{FS}, 
any compact analytic subset of the coarse moduli space of 
\textit{smooth} KE Fano manifolds with \textit{discrete} automorphism groups 
(constructed in \cite{Od2}) is projective. 
However, to treat general case, there are two main 
technical difficulties which are 
the presence of non-discrete automorphism groups 
(involving K-\textit{semi}stable varieties) and the log-terminal singularities. 

(Added in the revision of March, 2015:) Two and a half 
months after when the first manuscript of this paper 
appears, \cite{LWX2} made an announcement of a partial 
progress along this line and 
claimed the quasi-projectivity of the open locus $M$ of $\bar{M}$.



\end{document}